\documentclass[12pt]{amsart}
\usepackage{amsmath,amssymb,amsthm,graphics,amscd,mathrsfs}
\usepackage[colorlinks, breaklinks, linkcolor=blue]{hyperref}
\usepackage{breakurl}
\usepackage{array}
\usepackage{latexsym}
\usepackage{enumerate}
\usepackage[margin=1in]{geometry}
\usepackage{tikz,forest}

\newcommand\CA{{\mathcal A}}
\newcommand\CB{{\mathcal B}}

\newcommand\BBC{{\mathbb C}}

\newcommand\BBR{{\mathbb R}}

\newcommand\codim{\operatorname{codim}}

\newcommand\id{{\textrm{id}}}
\newcommand\Ind{{\operatorname{Ind}}}
\newcommand\GL{\operatorname{GL}}

\newcommand\inverse{^{-1}}
\renewcommand\th{{^{\text{th}}}}

\newcommand\Av{\operatorname{Av}}
\newcommand\spn{\operatorname{span}}

\numberwithin{equation}{section}

\theoremstyle{plain}
\newtheorem{lemma}[equation]{Lemma}
\newtheorem{theorem}[equation]{Theorem}

\newtheorem{corollary}[equation]{Corollary}

\thanks{The authors would like to thank their charming wives for their
  unwavering support during the preparation of this paper.}



\title [Invariants of the Cohomology of Complements of Coxeter arrangements]
{On the Invariants of the Cohomology of Complements of Coxeter arrangements}

\dedicatory{Dedicated to Michel Brou\'e}

\author[J.M. Douglass]{J. Matthew Douglass}
\address{Division of Mathematical Sciences,
National Science Foundation,
2415 Eisenhower Ave,
Alexandria, VA 22314, USA}
\email{mdouglas@nsf.gov}

\author[G. Pfeiffer]{G\"otz Pfeiffer} \address{School of Mathematics,
  Statistics and Applied Mathematics, National University of Ireland,
  Galway, University Road, Galway, Ireland}
\email{goetz.pfeiffer@nuigalway.ie}

\author[G. R\"ohrle]{Gerhard R\"ohrle} \address {Fakult\"at f\"ur
  Mathematik, Ruhr-Universit\"at Bochum, D-44780 Bochum, Germany}
\email{gerhard.roehrle@rub.de}

\begin{document}

\begin{abstract}
  We refine Brieskorn's study of the cohomology of the complement of
  the reflection arrangement of a finite Coxeter group $W$.  As a result we
  complete the verification of a conjecture by Felder and Veselov that
  gives an explicit basis of the space of $W$-invariants in this
  cohomology ring.
\end{abstract}

\maketitle
\allowdisplaybreaks

%%%%%%%%%%%%%%%%%%%%%%%%%%%%%%%%%%%%%%%%%%%%%%%%%%%%%%%%%%%%%%%%%%%%%%
%%%%%%%%%%%%% article body...
%%%%%%%%%%%%%%%%%%%%%%%%%%%%%%%%%%%%%%%%%%%%%%%%%%%%%%%%%%%%%%%%%%%%%%

%%%%%%%%%%%%%%%%%%%%%%%%%%%%%%%%%%%%%%%%%%%%%%%%%%%%%%%%%%%%%%%%%%%%%%
%%%%%%%%%%%%% \S1 Introduction
%%%%%%%%%%%%%%%%%%%%%%%%%%%%%%%%%%%%%%%%%%%%%%%%%%%%%%%%%%%%%%%%%%%%%%
\section{Introduction}

Suppose that $W$ is a finite Coxeter group with Coxeter generating set $S$
of size $|S|=l$. Let $V_{\BBR}$ be an $l$-dimensional, real vector space
that affords the reflection representation of $W$. Let $V=\BBC\otimes
_{\BBR} V_{\BBR}$ be the complexification of $V_{\BBR}$ and consider $W$ as
a subgroup of the group $\GL(V)$ of invertible $\BBC$-linear transformations
of $V$. Let $R$ denote the set of reflections in $W$. For each $r\in R$ let
$V^r$ denote the hyperplane of fixed points of $r$ in $V$, and set $\CA=\{\,
V^r\mid r\in R\,\}$. Then $(V, \CA)$ is the complexification of a Coxeter
arrangement.

The group $W$ acts naturally on the complement $M_W= V\setminus
\bigcup_{r\in R} V^r$ of the hyperplanes in $\CA$, and hence on the
cohomology of $M_W$ as algebra automorphisms. For $p\geq0$ let $H^p(M_W)$
denote the $p\th$ de Rahm cohomology space of $M_W$ with complex
coefficients and let $H^*(M_W)= \bigoplus_{p\geq 0} H^p(M_W)$ denote the
total cohomology of $M_W$. Felder and Veselov \cite{felderveselov:coxeter}
have conjectured an explicit construction of $H^p(M_W)^W$, the space of
$W$-invariants in $H^p(M_W)$, in terms of so-called special
involutions. They have verified their conjecture for all Coxeter groups
except those with irreducible components of type $E_7$, $E_8$, $F_4$, $H_3$,
or $H_4$.

In this note we complete the proof of the conjecture of Felder-Veselov by
reducing the problem to a computation in $H^{l}(M_W)$, implementing this
computation in the computer algebra system GAP3 (for any $W$), and then
performing the required calculations (for the remaining exceptional groups);
thus verifying that the conjecture is true. In \S\ref{sec:2} we give more
background and state the main result, and in \S\ref{sec:3} we describe some
novel algorithmic aspects of the implementation of the relations for the
Orlik-Solomon algebra of $W$ used to complete the calculations on which the
main results rely.

%%%%%%%%%%%%%%%%%%%%%%%%%%%%%%%%%%%%%%%%%%%%%%%%%%%%%%%%%%%%%%%%%%%%%%
%%%%%%%%%%%%% \S2 Background and main results
%%%%%%%%%%%%%%%%%%%%%%%%%%%%%%%%%%%%%%%%%%%%%%%%%%%%%%%%%%%%%%%%%%%%%%

%%%%%%%%%%%%%%%%%%%%%%%%%%%%%%%%%%%%%%%%%%%%%%%%
\section{Background and main results} \label{sec:2}
%%%%%%%%%%%%%%%%%%%%%%%%%%%%%%%%%%%%%%%%%%%%%%%%

In order to state the Felder-Veselov conjecture precisely, we need
additional notation.

%%%%%%%%%%%%%%%%%%%%%%%%%%%%%%%%%%%%%%%%%%%%%%%%
\subsection{The Felder-Veselov conjecture}

Let $\Phi \subseteq V_{\BBR}$ be a root system for $W$ as in
\cite [\S1.1] {geckpfeiffer:characters} and let $\{\, \alpha_s\mid s\in
S\,\}$ be the roots corresponding to elements of $S$.

Orlik and Solomon \cite{orliksolomon:combinatorics} give a
combinatorial presentation of the cohomology algebra $H^*(M_W)$ that is
suitable for machine computation. The \emph{Orlik-Solomon algebra of $W$} is
the $\BBC$-algebra $A(W)$ with generators $\{\, a_r\mid r\in R\}$ and
relations
\begin{itemize}
\item $a_{r} a_{s}= -a_{s} a_{r}$ for $r,s\in R$ and
\item whenever $\{V^{r_1}, \dots, V^{r_p}\}$ is linearly dependent for
  $r_1, \dots, r_p \in R$, we have
  \[
  \sum_{i=1}^p (-1)^i a_{r_1} \dotsm \widehat{ a_{r_i}} \dotsm a_{r_p} =0,
  \]
\end{itemize}
where the notation $\widehat{ a_{r_i}}$ indicates omission of the term $a_{r_i}$.
The algebra $A(W)$ is
naturally graded with $A^p(W)$ equal to the span of all $a_{r_1} \dotsm
a_{r_p}$ such that $\codim V^{r_1}\cap \dots\cap V^{r_p}= p$. The rule $(w,
a_r)\mapsto a_{wrw\inverse}$, for $w\in W$ and $r\in R$, extends to an
action of $W$ on $A(W)$ as degree-preserving, algebra automorphisms. Orlik
and Solomon show that the rule $a_s\mapsto d\alpha_s^{\vee}/
\alpha_s^{\vee}$, where $s\in S$ and $\alpha_s^{\vee} \in V^*$ denotes the
extension of the coroot $\alpha_s^{\vee}\in V_{\BBR}^*$, extends to a
$W$-equivariant isomorphism of graded $\BBC$-algebras $A(W) \cong
H^*(M_W)$. See \cite{orlikterao:arrangements} for details. In the following
we work with the Orlik-Solomon algebra of $W$.

Fix an arbitrary linear order on $R$, say $R = \{\, r_1, \dots, r_n\, \}$.
For any subset $T = \{\, r_{i_1}, \dots, r_{i_p} \,\}$ of $R$
with $i_1 < \dots < i_p$, define
\[
  a_T := a_{i_1} \cdots a_{i_p} \in A(W).
\]

For an involution $t \in W$ there is a direct sum decomposition $V_{\BBR}
\cong V_1 \oplus V_{-1}$, where $V_1$ and $V_{-1}$ are the $1$- and
$(-1)$-eigenspaces of $t$, respectively. Define $\Phi_{1} = \Phi \cap V_{1}$
and $\Phi_{-1} = \Phi \cap V_{-1}$. Following \cite{felderveselov:coxeter}, we
say that $t$ is \emph{special}, if for any root $\alpha \in \Phi$ at least
one of its projections onto $V_{1}$ or $V_{-1}$ is proportional to a root in
$\Phi_{1}$ or $\Phi_{-1}$, respectively. Clearly, this definition does not
depend on the choice of root system for $W$.

Suppose $t$ is a special involution and $p=\dim V_{-1}$. Then $\Phi_{-1}$ is
a root system in $V_{-1}$. Choose a base of $\Phi_{-1}$ and let $S(t)\subseteq
R$ be the reflections in $W$ corresponding to the roots in this base.
Then $a_{S(t)} \in A^p(W)$.
Let
\[
  \Av\colon A(W) \to A(W)^W
\]
be the averaging map,
where $A(W)^W$ is the space of $W$-invariants of $A(W)$. Felder and
Veselov \cite{felderveselov:coxeter} make the following conjecture that we
state as a theorem.

\begin{theorem} \label{thm:fvc} Let $W$ be a finite Coxeter group.  Then
  \begin{enumerate}
  \item for any special involution $t$ of $W$, the element $\Av(a_{S(t)}) \in A(W)^W$ is non-zero, 
  and \label{it:1}
  \item any element in $A(W)^W$ is a linear combination of elements $\Av(a_{S(t)})$ from \eqref{it:1}. \label{it:2}
  \end{enumerate}
\end{theorem}

More precisely, if $m$ is the number of conjugacy classes of special
involutions in $W$, then it follows from the theorem
(and the observation that, for $w \in W$, up to a sign
$\Av(a_{S(wtw\inverse)}) = \Av(a_{S(t)})$) that $\dim
A(W)^W \leq m$. For each irreducible finite Coxeter group $W$,
Brieskorn \cite[Thm.~7]{brieskorn:tresses} has computed the
Betti numbers of the manifold $M_W/W$ and thus the Poincar\'e polynomial of
$A(W)^W$. It turns out that $\dim A(W)^W=m$ and so the next corollary is an
immediate consequence of the theorem.

\begin{corollary}\label{cor:fvc}
  Suppose $W$ is a finite Coxeter group and $\{t_1, \dots, t_m\}$ is a set
  of representatives of the conjugacy classes of special involutions in
  $W$. Then $\{\Av(a_{S(t_1)}), \dots, \Av(a_{S(t_m)}) \}$ is a basis of
  $A(W)^W$.
\end{corollary}

A proof of Theorem \ref{thm:fvc} is given in the next section. Roughly
speaking, the proof of \eqref{it:1} consists of reducing the assertion to
the case when $t$ is the longest element in $W$. This statement is then
checked case-by-case, using GAP3 for the exceptional groups. The assertion
in \eqref{it:2} follows from an inspection of the reduction used to prove
\eqref{it:1}.

%%%%%%%%%%%%%%%%%%%%%%%%%%%%%%%%%%%%%%%%%%%%%%%%
\subsection{A reduction}

The reduction of Theorem \ref{thm:fvc}\,\eqref{it:1} to the case of longest
words in top degree is based on a decomposition of the representation of $W$
on $A^p(W)$ as a sum of induced representations due to Lehrer and
Solomon~\cite{lehrersolomon:symmetric}.

Each subset $I$ of $S$ determines
\begin{itemize}
\item a standard parabolic subgroup $W_I$ of $W$ generated by $I$,
\item subspaces $V_I= \spn\{\,\alpha_s\mid s\in I\,\}$ and
  $X_I=\bigcap_{s\in I} V^s$ of $V$ such that $V\cong V_I \oplus X_I$, and
\item a subspace $A(W)_I= \spn \{\, a_{r_1} \dotsm a_{r_d} \mid V^{r_1}\cap
  \dots\cap V^{r_d} =X_I\,\} \subseteq A^{|I|}(W)$.
\end{itemize}
Let $N_I=N_W(W_I)$ be the normalizer of $W_I$ in $W$. It is easy to see that
$A(W)_I$ is an $N_I$-stable subspace of $A(W)$.

It is known that for subsets $I$ and $J$ of $S$, the following are
equivalent:
\begin{enumerate}
\item $W_I$ and $W_J$ are conjugate, and
\item $X_I$ is a $W$-translate of $X_J$.
\end{enumerate}
This motivates the notion of \emph{shapes} of $W$ that index the
Lehrer-Solomon decomposition \eqref{eq:decomp} of $A(W)$ as follows.
For $I, J\subseteq S$, define $I\sim J$ if $J=wIw\inverse$ for some $w$ in
$W$. This defines an equivalence relation on the power set of $S$.  A
\emph{shape} (for $W$) is a $\sim$-equivalence class. Let $\Lambda$ denote the
set of shapes and for each $\lambda\in \Lambda$, fix once and for all a
representative $I_\lambda \in \lambda$ and set $l_\lambda = |I_\lambda|$.

Lehrer and Solomon \cite[\S2]{lehrersolomon:symmetric} have shown that the
representation of $W$ on $A^p (W)$ decomposes as a direct sum of induced
representations:
\begin{equation}
  \label{eq:decomp}
  A^p (W) \cong \bigoplus_{\substack{\lambda\in \Lambda\\ l_\lambda=p}}
  \Ind_{N_{I_\lambda}}^W A(W)_{I_\lambda} .
\end{equation}

Notice that for $I\subseteq S$, $W_I$ is a Coxeter group with Coxeter
generating set $I$, and that $V_I$ is the complexification of the reflection
representation of $W_I$. Thus, we may consider the Orlik-Solomon algebra
$A(W_I)$ of $W_I$.  Clearly the action of $W_I$ on $A(W_I)$ extends to an
action of $N_I$, and so in particular to a representation of $N_I$ on the
top component $A^{|I|} (W_I)$. It follows easily from a standard property of
Orlik-Solomon algebras (see \cite[\S 3.1, Cor.~6.28]
{orlikterao:arrangements}) that there is an $N_I$-equivariant isomorphism
$A(W)_I \cong A^{|I|}(W_I)$. Therefore, summing over $p$, the decomposition
\eqref{eq:decomp} and Frobenius reciprocity yield
\begin{equation}
  \label{eq:dinv}
  A(W)^W=  A^l (W)^{W} + \sum_{\lambda\in
    \Lambda \setminus \{S\}} (A (W)_{I_\lambda}) ^{N_{I_\lambda}}
  \cong  A^l (W)^{W} \oplus \bigoplus_{\lambda\in
    \Lambda \setminus \{S\}} A^{l_\lambda} (W_{I_\lambda}) ^{N_{I_\lambda}}.
\end{equation}

The decomposition \eqref{eq:dinv} reduces the computation of $A(W)^W$ to
that of $A^l(W)^W$ and $A^{|I|} (W_I)^{N_I}$, for $I$ a proper subset of
$S$. We show below that the non-zero summands are indexed by the set of
conjugacy classes of special involutions, that each non-zero summand is
one-dimensional, and that this decomposition is just the decomposition of
$A(W)^W$ into one-dimensional subspaces given by the basis in Corollary
\ref{cor:fvc}.

%%%%%%%%%%%%%%%%%%%%%%%%%%%%%%%%%%%%%%%%%%%%%%%%
\subsection{Top degree invariants}
Consider first the summand $A^l (W)^{W}$ in \eqref{eq:dinv}. Following
Richardson \cite{richardson:involutions}, we say that a subset $I\subseteq S$
satisfies the $(-1)$-condition if $W_I$ contains an element that acts as
$-1$ on $V_I$. Let $w_I$ denote the longest element in $W_I$ with respect to
the length function determined by $S$. Obviously $I$ satisfies the
$(-1)$-condition if and only if each irreducible factor of $W_I$ does.
In addition, it is straightforward to check that, for
$W_I$ irreducible, $I$ satisfies the $(-1)$-condition if and only if
$w_I$ is equal to $-\id_{V_I}$ (see \cite[\S1]{richardson:involutions}). It
follows that in general, $I$ satisfies the $(-1)$-condition if and only if
$w_I$ is equal to $-\id_{V_I}$.

Taking $I=S$, it is clear that if $S$ satisfies the $(-1)$-condition, then
$w_S$ is a special involution in $W$. Conversely, Felder and
Veselov~\cite{felderveselov:coxeter} have
observed that if $W$ is irreducible and $S$ does not satisfy the
$(-1)$-condition, then $w_S$ is a not a special involution. It is immediate
from the definition that an involution $t\in W$ is special if and only if
the components of $t$ in each irreducible factor of $W$ are special. It
follows that in general, $S$ satisfies the $(-1)$-condition if and only if
$w_S$ is a special involution in $W$.

As noted above, Brieskorn has computed the Poincar\'e polynomials of the
graded vector spaces $A(W)^W$ for all irreducible $W$. It follows from this
computation that $\dim A^l(W)^W=1$ or $0$ according as to whether or not $S$
satisfies the $(-1)$-condition. It follows that in general, $S$ satisfies
the $(-1)$-condition if and only if $A^l(W)^W\ne 0$, and if so, then
$A^l(W)^W$ is one-dimensional.

To summarize, the following are equivalent for any finite Coxeter group:
\begin{itemize}
\item $A^l(W)^W\ne 0$.
\item The longest element in $W$ acts as minus the identity in the
  reflection representation.
\item The longest element in $W$ is a special involution.
\end{itemize}

Notice that for $I\subseteq S$, we have $S(w_I) = I$ and
hence $a_I= a_{S(w_I)}$.
We can now state our main theorem.

\begin{theorem}
  \label{thm:invariants1}
  Suppose $W$ is a finite Coxeter group with Coxeter generating set
  $S$ of size $|S|=l$. The following are equivalent:
  \begin{enumerate}
  \item $A^l(W)^W\ne 0$. \label{it:invit1}
  \item The longest element in $W$ is a special
    involution. \label{it:invit2}
  \item $\Av(a_S)\ne 0$. \label{it:invit3}
  \end{enumerate}
  If these conditions hold, then $A^l(W)^W$ is one-dimensional with
  generator $\Av (a_S)$.
\end{theorem}

\begin{proof}
  The equivalence of \eqref{it:invit1} and \eqref{it:invit2} is explained
  above, and it is clear that if $\Av (a_S)\ne 0$, then
  $A^l(W)^W\ne 0$.  Thus, it remains to show that if $w_S$ is a special
  involution, then $\Av (a_S)\ne 0$. It follows from the
  preceding discussion that without loss of generality we may assume that
  $W$ is irreducible. Then $w_S$ is a special involution if and only if $W$
  is of type $A_1$, $B_n$, $D_{2n}$, $F_4$, $E_7$, $E_8$, $H_3$, $H_4$, or
  $I_2(2n)$. Felder and Veselov \cite{felderveselov:coxeter} have
  established the statement for all types other than $E_7$, $E_8$, $F_4$,
  $H_3$, and $H_4$. We have checked these remaining instances by machine
  computations. The most challenging cases are when $W$ is of type $E_7$ and
  $E_8$, requiring sophisticated programming techniques and intricate
  reductions to deal with the relations in $A(W)$. Details regarding the
  implementation of these computations are given in the next section.
\end{proof}

The summands in \eqref{eq:dinv} not equal $A^l(W)^W$ are
described in the next lemma.

\begin{lemma}\label{lem:topI}
  Let $I\subseteq S$ with $|I|=p$ and consider $A^p(W_I)^{N_I}$.
  \begin{enumerate}
  \item Suppose $I$ does not satisfy the $(-1)$-condition. Then
    $A^p(W_I)^{N_I}=0$.
  \item Suppose $I$ satisfies the $(-1)$-condition and $w_I$ is a not
    special involution in $W$. Then $A^p(W_I)^{N_I}=0$.
  \item Suppose $I$ satisfies the $(-1)$-condition and $w_I$ is a special
    involution in $W$. Then $A^p(W_I)^{N_I}$ is one-dimensional and
    $\Av(a_I) \ne 0$.
  \end{enumerate}
\end{lemma}

\begin{proof}
  By Theorem \ref{thm:invariants1} we may assume that $I$ is a proper subset
  of $S$.

  If $I$ does not satisfy the $(-1)$-condition, then $A^{p} (W_{I})
  ^{W_{I}}=0$, by Theorem \ref{thm:invariants1}, and $A^{p} (W_{I}) ^{N_{I}}
  \subseteq A^{p} (W_{I}) ^{W_{I}}$, so $A^{p} (W_{I}) ^{N_{I}}=0$.

  In order to handle the cases when $I$ does satisfy the
  $(-1)$-condition, we need to recall some facts about the structure of the
  normalizer $N_I$ due to Howlett and Pfeiffer-R\"ohrle. First, Howlett
  \cite{howlett:normalizers} has shown that $W_I$ has a canonical complement
  in $N_I$, denoted here by $C_I$. Second, Pfeiffer and R\"ohrle
  \cite{pfeifferroehrle:special} have shown, under the assumption that
  $I$ satisfies the $(-1)$-condition, $w_I$ is a special
  involution in $W$ if and only if $C_I$ centralizes $W_I$.

  Now suppose $I$ satisfies the $(-1)$-condition and $w_I$ is a not special
  involution in $W$. Then $A^{p} (W_{I}) ^{N_{I}} \subseteq A^{p} (W_{I})
  ^{W_{I}}$, $A^{p} (W_{I}) ^{W_{I}}$ is one-dimensional with generator
  $\Av_{I}(a_I)$, where $\Av_{I} : A(W_I) \to A(W_I)^{W_I}$ denotes the averaging map for $W_I$,
  and $C_I$ does not centralize $W_I$.
  We may assume that $W$ is
  irreducible. Then it follows from the classification of
  irreducible finite Coxeter groups that $W_I$ has at most one component not
  of type $A$, and because $I$ satisfies the $(-1)$-condition, each
  component of type $A$ also satisfies the $(-1)$-condition and so is of
  type $A_1$. Moreover, the component not of type $A$ must be of type $B_k$
  ($k\geq2$), $D_{2k}$ ($k\geq2$), $E_7$, or $H_3$. Considering these
  possibilities case-by-case using the description of $C_I$ in
  \cite{howlett:normalizers}, it can be checked that in all cases when $C_I$
  does not centralize $W_I$, the group $C_I$ contains an element $c$ that
  acts on $W_I$ as a graph automorphism that transposes two nodes of the
  Coxeter graph of $W_I$ and leaves the other nodes fixed.

  Thus the relations of $A(W)$ yield
  $c(a_I) = -a_I$ and so
  \[
  c\big(\Av_{I}( a_I) \big)= \Av_{I}(c( a_I)) =
  -\Av_{I}( a_I),
  \]
  showing that $\Av_{I}(a_I)$ is not invariant under $N_I$.
  Consequently,  $A^p(W_I)^{N_I} \ne A^p(W_I)^{W_I}$, whence $A^p(W_I)^{N_I} = 0$
  as $\dim A^p(W_I)^{W_I} = 1$.

  Finally, suppose $I$ does satisfy the $(-1)$-condition and $w_I$ is a special
  involution in $W$. Then $A^{p} (W_{I}) ^{N_{I}} \subseteq A^{p} (W_{I})
  ^{W_{I}}$, $A^{p} (W_{I}) ^{W_{I}}$ is one-dimensional with generator
  $\Av_{I}(a_I)$, and $C_I$ centralizes $W_I$. Hence, for all $c\in C_I$,
  $c\big(\Av_{I}(a_I) \big)= \Av_{I}(a_I)$
  and so $A^{p} (W_{I}) ^{N_{I}} =A^{p} (W_{I}) ^{W_{I}} \ne0$. To
  complete the proof, let $Y\subseteq W$ be a complete set of left
  $N_I$-coset representatives in $W$. Then
  \[
  \Av(a_I) = |Y|\inverse \sum_{y\in Y} y\big(
  \Av_{I}(a_I) \big) \in \sum_{y\in Y} y\big( A(W)_I \big).
  \]
  But now the sum $\sum_{y\in Y} y\big( A(W)_I \big)$ in \eqref{eq:dinv} is
  direct and $y\big( \Av_{I}(a_I) \big) \in y\big( A(W)_I \big)$
  for $y\in Y$, so $\Av(a_I) \ne 0$.
\end{proof}

%%%%%%%%%%%%%%%%%%%%%%%%%%%%%%%%%%%%%%%%%%%%%%%%
\subsection{Proof of Theorem \ref{thm:fvc}} Richardson
\cite{richardson:involutions} has shown that $t\in W$ is an involution if
and only if there is a subset $I\subseteq S$ that satisfies the
$(-1)$-condition such that $w_I$ is conjugate to $t$. Therefore, if $t$ is a 
special involution, there is a subset $I\subseteq S$ that satisfies the
$(-1)$-condition such that $t$ is conjugate to $w_I$. But then $w_I$ is a
special involution and $\Av(a_{S(t)}) = \pm \Av(a_I)$, so it
follows from Theorem \ref{thm:invariants1} and Lemma \ref{lem:topI} that
$\Av(a_{S(t)}) \ne 0$.

Finally, it follows from the decomposition \eqref{eq:dinv}, Theorem
\ref{thm:invariants1}, and Lemma \ref{lem:topI}, that $A(W)^W$ is spanned by
the elements $\Av(a_I)$ where $I$ runs over the subsets of $S$ that
satisfy the $(-1)$-condition and for which $w_I$ is a special
involution. More precisely, if $\Lambda_{-1}$ denotes the set of shapes consisting of
subsets that satisfy the $(-1)$-property and $\Lambda_1$ denotes the set of shapes
consisting of subsets $I$ such that $C_I$ centralizes $W_I$, then $\Lambda_{-1} \cap
\Lambda_1$ indexes the set of conjugacy classes of special involutions and $\{\,
\Av(a_{I_\lambda}) \mid \lambda\in \Lambda_{-1} \cap \Lambda_1\,\}$ is a basis of
$A(W)^W$.

%%%%%%%%%%%%%%%%%%%%%%%%%%%%%%%%%%%%%%%%%%%%%%%%%%%%%%%%%%%%%%%%%%%%%%
%%%%%%%%%%%%% \S Computational and algorithmic aspects
%%%%%%%%%%%%%%%%%%%%%%%%%%%%%%%%%%%%%%%%%%%%%%%%%%%%%%%%%%%%%%%%%%%%%%

\section{Computational and algorithmic aspects}
\label{sec:3}

We have implemented the relations for the Orlik-Solomom algebra $A(W)$ with
the use of the computer algebra system GAP3 \cite{gap3} and the CHEVIE
package \cite{chevie}. The papers
\cite{bishopdouglasspfeifferroehrle:computations},
\cite{bishopdouglasspfeifferroehrle:computationsII}, and
\cite{bishopdouglasspfeifferroehrle:computationsIII} contain some of the
details of this implementation. In this section, we describe some
refinements of our earlier techniques that allow us to complete the
computations used in the proof of Theorem \ref{thm:invariants1}.

%%%%%%%%%%%%%%%%%%%%%%%%%%%%%%%%%%%%%%%%%%%%%%%%
\subsection{The broken circuit bases of $A(W)$}
The broken circuit bases of $A(W)$ is a computationally efficient basis to
use for machine calculations for individual Coxeter groups that is
compatible with the decomposition of $A(W)$ arising from (\ref{eq:decomp}).
For later reference we briefly recall the construction of this basis.

Recall the fixed linear order on $R=\{\, r_1, \dots, r_n\,\}$.
Recall that a subset $T\subseteq R$ is \emph{independent} if $\codim
(\bigcap_{r\in T} V_{\BBR}^r) = |T|$ and \emph{dependent} otherwise. A
\emph{circuit} is a subset of $R$ that is minimally linearly dependent. That
is, it is linearly dependent, but any proper subset is linearly
independent. A \emph{broken circuit} is a subset of $R$ that is obtained
from a circuit by deleting the maximal element with respect to the fixed
linear order on $R$. Thus, broken circuits are subsets of the form
$\{r_{i_1}, \dots, r_{i_p} \}$ where there is a $j>i_p$ so that $\{ r_{i_1},
\dots, r_{i_p}, r_j \}$ is a circuit. A subset of $R$ is
\emph{$\chi$-independent} if it does not contain a broken circuit.

It is convenient to identify $R=\{r_1, \dots, r_n\}$ with the set $\{1,
\dots, n\}$ and to identify ordered subsets of $R$ with words in the
alphabet $\{1, \dots, n\}$. If $T= i_1\cdots i_p$ is such a word, then
adjectives applied to $\{r_{i_1}, \dots, r_{i_p}\}$ are also applied to
$T$. For example, $T = i_1\cdots i_p$ is \emph{independent} if the subset
$\{r_{i_1}, \dots, r_{i_p} \}$ of $R$ is independent.

Write $a_i$ instead of $a_{r_i}$ for the corresponding algebra generator of
$A(W)$. Given a word $T= i_1\cdots i_p$ of positive integers less than or
equal $n$, define an element, $a_T$, in $A^p(W)$ by $a_T= a_{i_1} \dotsm
a_{i_p}$
(in analogy with the definiton of $a_T$ for a \emph{subset} $T$ of $R$ in
Section~\ref{sec:2}). Let $\CB$ denote the set of all $\chi$-independent words $i_1
\cdots i_p$ such that $i_1<\dots <i_p$. It is shown in \cite[\S
3.1]{orlikterao:arrangements} that $\{\, a_T\mid T\in \CB \,\}$ is a basis
of $A(W)$, called there a \emph{broken circuit} basis. A broken circuit
basis is a common basis for the subspaces $A(W)^p$ and $A(W)_I$ of $A(W)$
and is compatible with the isomorphisms $A(W)_I \cong A^{|I|}(W_I)$ for
$I\subseteq S$.

When working in GAP3 it is more convenient to let groups act on the
right. Thus, in this section we consider the right action of $W$ on $A(W)$
that satisfies $a_T. w= a_{T.w}$, where if $T= i_1 \cdots i_p$, then $T.w =
j_1 \cdots j_p$, where $w\inverse r_{i_1} w =r_{j_1}$, \dots, $w\inverse
r_{i_p}w = r_{j_p}$. Let $|T.w| = j_1' \cdots j_p'$ be a rearrangement of
$T.w$ in increasing order and let $\epsilon (T,w)$ be the sign of a
permutation that is needed to sort the word $T.w$ in increasing order. Then
$a_T. w= a_{T.w} = \epsilon (T,w) a_{|T.w|}$.

For $a \in A(W)$, let us denote by $\overline{a}$ the coordinate vector of
$a$ with respect to the broken circuit basis $\{a_T \mid T \in
\mathcal{B}\}$ of $A(W)$, i.e., an explicit list of coefficients $\beta_T
\in \BBC$ such that $a = \sum_T \beta_T a_T$.  In the application, most
coefficients $\beta_T$ are zero and the list can be stored as a sparse list
consisting of the non-zero coefficients only.

The proof of Theorem \ref{thm:invariants1} boils down to computing
\[
\omega= a_S.\sum_{w \in W} w = \sum_{w \in W} a_S.w.
\]
The task of checking whether $\omega \neq 0$ reduces to
\begin{enumerate}
\item computing the image $a_S.w$ of $a_S$ under each group element $w \in
  W$,
\item expressing each image $a_S.w$ as $\overline{a_S.w}$ in terms of the
  broken circuit basis,
\item computing $\overline{\omega} = \sum_{w \in W} \overline{a_S.w}$.
\end{enumerate}
While this looks straightforward (and in the case of small groups $W$ it is
straightforward), it can be challenging for higher rank Coxeter groups of
exceptional type, i.e., for $E_7$ and $E_8$. The difficulties arise from
\begin{itemize}
\item the order of $W$ and hence the number of images $a_S.w$ that need to
  be determined,
\item the need to explicitly express an element $a_T$ for arbitrary subsets
  $T$ of $R$ as a linear combination $\overline{a_T}$ of the broken circuit
  basis,
\item the need to efficiently represent the $|W|$ elements of the broken
  circuit basis of $A(W)$.
\end{itemize}
We address all these points in turn in the following subsections.

%%%%%%%%%%%%%%%%%%%%%%%%%%%%%%%%%%%%%%%%%%%%%%%%
\subsection{Decomposing $\mathbf W$}
\label{sec:decomposing-w}

In all cases, it turns out that the element $\omega$ has tiny
support in the broken circuit basis of $A(W)$ relative to the size of $W$.
In contrast, this need not be the case
for intermediate results, and the time
it takes to compute $\overline{\omega}$ depends subtly on the order in which
various steps are taken. We chose to separate the calculation as follows.

The standard parabolic subgroup $W_J$ has a distinguished set $D$ of left
coset representatives, consisting of the unique elements of minimal length
in their coset.  As each element $w \in W$ has a decomposition $w = x \cdot
w'$ for uniquely determined elements $x \in D$, $w' \in W_J$, in the group
algebra of $W$ we can write
\[
\sum_{w \in W} w = \Bigl(\sum_{x \in D} x\Bigr) \cdot \Bigl(\sum_{w' \in
  W_J} w'\Bigr).
\]
In fact, there are parabolic subgroups
\[
\{1\} = W_0 < W_1 < \dots < W_l = W,
\]
such that $W_{j-1}$ is a maximal standard parabolic subgroup of $W_j$, and
$W_j = D_j W_{j-1}$ for the distinguished set $D_j$ of left coset
representatives of $W_{j-1}$ in $W_j$, for $j = 1, \dots, l$. Thus each
element $w \in W$ can be written as
\[
w = x_l \cdots x_2 \cdot x_1,
\]
for uniquely determined elements $x_j \in D_j$, $j = 1, \dots, l$. Hence, in
the group algebra of~$W$,
\[
\sum_{w \in W} w = \Bigl(\sum_{x_l \in D_l} x_l\Bigr) \cdots \Bigl(\sum_{x_2
  \in D_2} x_2\Bigr) \cdot \Bigl(\sum_{x_1 \in D_1} x_1\Bigr),
\]
and we can compute
\[
\omega = \Bigl(\Bigl(\dotsm\Bigl(a_S.\sum_{x_l \in D_l} x_l\Bigr)
\dotsm\Bigr) \cdot \sum_{x_2 \in D_2} x_2\Bigr) \cdot \sum_{x_1 \in D_1}
x_1.
\]
In this way, instead of $\prod |D_j| = |W|$, only $\sum |D_j|$ images of
algebra elements under group elements need to be computed and converted into
the basis.

In practice, we use the chain of parabolics induced by the labelling
of generators $S = \{s_1, \dots, s_l\}$ in CHEVIE, with $W_j =
\langle s_1, \dots, s_j \rangle$, $j = 1, \dots,l$.
In the case of $E_8$ with Coxeter diagram
\begin{center}
  \begin{tikzpicture}[scale=0.6]
    \node (1) at (0,0) {$_1$};
    \node (2) at (2,1) {$_2$};
    \node (3) at (1,0) {$_3$};
    \node (4) at (2,0) {$_4$};
    \node (5) at (3,0) {$_5$};
    \node (6) at (4,0) {$_6$};
    \node (7) at (5,0) {$_7$};
    \node (8) at (6,0) {$_8$};
    \draw (1) -- (3) -- (4) -- (5) -- (6) -- (7) -- (8);
    \draw (2) -- (4);
  \end{tikzpicture},
\end{center}
this reduces the number of image
calculations from a formidable $|W| = 696,729,600$ to a mere $\sum |D_j| =
356$.  However, the algebra elements now are linear combinations of words,
rather than just words.

Set $q_j = a_S.\sum_{x_l \in D_l} x_l \dotsm \sum_{x_j \in D_j} x_j$,
for $j = 1, \dots, l$.  Then $\omega = q_1$. Now $\overline{\omega}$ is
computed in $l$ steps, for $j$ from $l$ down to $1$, as follows.  Assuming
that $\overline{q_{j+1}}$ is known, one obtains
\[
\overline{q_j} = \overline{\overline{q_{j+1}}.\sum_{x \in D_j}x} = \sum_{x
  \in D_j} \overline{\overline{q_{j+1}}.x}.
\]
Here, if $\overline{q_{j+1}} = \sum_T \beta_T a_T$, then
\[
\overline{\overline{q_{j+1}}.x} = \sum_T \beta_T \overline{a_T.x} = \sum_T
\beta_T\, \epsilon(T,x)\, \overline{a_{|T.x|}}.
\]
Initially, this requires us to compute $\overline{a_S}$. For this, it turns
out to be convenient to choose an order on $R$ that makes $a_S$ a basis
element, or at least close to one.

%%%%%%%%%%%%%%%%%%%%%%%%%%%%%%%%%%%%%%%%%%%%%%%%
\subsection{Rewrite Rules.}

In order to express arbitrary elements of $A(W)$ in terms of the broken
circuit basis, we need to be able to express an element $a_T$, for an
arbitrary word $T = i_1\cdots i_p$ with $i_1< \cdots <i_p$, in terms of the
broken circuit basis, i.e., to compute the coefficients of $\overline{a_T}$.
First we note that the broken circuit basis has the following useful
\emph{Schreier property}: if $T = i_1\cdots
i_p$ is in $\CB$, then $T'$ is in $\CB$ for any \emph{prefix} $T' = i_1\dots
i_k$ of $T$ ($k \leq p$). Thus, if $T$ is a strictly increasing word and
$T'$ is a proper prefix of $T$ such that $a_{T'}$ is not a basis element,
then neither is $a_T$.  Using the relations in $A(W)$, we compute
$\overline{a_T}$ as follows.
\begin{enumerate}
\item Find the minimal $k$ such that $i_1\cdots i_k$ contains a broken
  circuit.  If no such $k$ exists, then $a_T$ is a basis element of $A(W)$ (by
  definition).
\item Otherwise, find the maximal index $u$ such that $i_1\cdots i_k \,u$ is
  a circuit (such $u$ exists, is larger than $i_k$, and can easily be
  identified by computing the rank of the corresponding matrix of root
  vectors).
\item If $u$ occurs in  $T$, then $a_T = 0$.  Otherwise, using the relations in
  $A(W)$,
  \[
  i_1 i_2 \cdots i_k = \sum_j (-1)^{k-j} (i_1 \cdots
  \hat{i}_j \cdots i_k) u
  \]
  and we can compute $\overline{a_T}$ recursively as
  \[
  \overline{a_T} = \sum_j (-1)^{k-j} \overline{(i_1 \cdots \hat{i}_j
    \cdots i_k) u (i_{k+1} \cdots i_p)}.
  \]
\end{enumerate}
This process must terminate since (in the lexicographic order of words in
$\{1, \dots, n\}$) all of the replacement terms on the right hand side are
strictly bigger than the original word $T$.

%%%%%%%%%%%%%%%%%%%%%%%%%%%%%%%%%%%%%%%%%%%%%%%%
\subsection{Constructing and Managing a Basis.}
The above procedure for expressing an element of $A(W)$ in the broken
circuit basis depends on an efficient procedure for distinguishing words of
$\CB$ from other words. The definition of a broken circuit basis is not
particularly well suited for this purpose: testing whether a subword of a
word $T$ is in $\CB$ in isolation is not straightforward, and the cost of
testing all subwords of $T$ is exponential.  This task can be carried out
more efficiently in the presence of some pre-computed data.  If, for
example, a complete list of words in $\CB$ is known, then deciding whether
an arbitrary increasing word $T$ is in $\CB$ or not is a simple lookup
operation.  However, as $|\CB| = |W|$, such a list is expensive to
compute and to store for larger groups.

Here, taking advantage of the Schreier property, we use a rooted, directed
acyclic graph $\Gamma$ on a small number of nodes to represent the words in
$\CB$.  The root node, labelled $0$, corresponds to the empty word.  All
other nodes are labelled by the positive integers indexing the generators
$a_1$, \dots, $a_n$ of $A(W)$. In this graph, directed paths starting at the
root node represent words in $\CB$. To decide whether a given word lies
in $\CB$, one simply traces the word (reading from left to right), beginning
at the root node, through the graph.  If at some point no edge leads to a
node labelled by the next letter, the corresponding prefix (and hence the
word) is not in $\CB$ (whereas all prefixes so far were in $\CB$).

The graph $\Gamma$ for $W$ of type $A_3$ with the reflections linearly
ordered by $s_{12}< s_{23}< s_{34} < s_{13}<s_{24}< s_{14}$ is given in
Figure \ref{fig:3}. For example the words $2$, $2\, 4$, and $2\, 4\, 6$ are
in $\CB$ and the word $2\, 4\, 5$ is not.

We construct such a graph as follows. First of all, it is useful to represent
the subsets of $\{1, \dots, n\}$ with at most $l+1$ elements as nodes in a rooted,
ranked tree $\Upsilon$, with nodes labeled by $\{0, 1, \dots, n\}$. Here,
the root node, labeled $0$, corresponds to the empty set and has rank
$0$. Each other node represents the subset consisting of the labels of the
nodes along the unique path back to the root node (excluding the root
node). Figure \ref{fig:1} shows this tree for the case when $l=3$ and
$n=6$. For example, the last circled node labelled by $5$ with rank $3$
represents the subset $\{3,4,5\}$, obtained from the path $0$-$3$-$4$-$5$.
The formal definition of $\Upsilon$ is easily extracted from this example.

\begin{figure}[htb]
  \centering
  \begin{forest}
    for tree={l=15mm,s sep=1mm,inner sep=0mm,minimum size=4.3mm}
    [$_0$ [$_1$ [$_2$, red, draw=red, circle [$_3$, red, draw=red, circle
    [$_4$, blue, draw=blue] [$_5$, blue, draw=blue] [$_6$, blue, draw=blue]
    ] [$_4$, blue, draw=blue [$_5$, blue, draw=blue] [$_6$, blue, draw=blue]
    ] [$_5$, red, draw=red, circle [$_6$, blue, draw=blue] ] [$_6$, red,
    draw=red, circle] ] [$_3$ [$_4$, red, draw=red, circle [$_5$, blue,
    draw=blue] [$_6$, blue, draw=blue] ] [$_5$, red, draw=red, circle [$_6$,
    blue, draw=blue] ] [$_6$] ] [$_4$ [$_5$, red, draw=red, circle [$_6$,
    blue, draw=blue] ] [$_6$] ] [$_5$, red, draw=red, circle [$_6$, blue,
    draw=blue] ] [$_6$] ] [$_2$ [$_3$, red, draw=red, circle [$_4$, red,
    draw=red, circle [$_5$, blue, draw=blue] [$_6$, blue, draw=blue] ]
    [$_5$, blue, draw=blue [$_6$, blue, draw=blue] ] [$_6$, red, draw=red,
    circle] ] [$_4$ [$_5$, red, draw=red, circle [$_6$, blue, draw=blue] ]
    [$_6$] ] [$_5$ [$_6$] ] [$_6$] ] [$_3$ [$_4$, red, draw=red, circle
    [$_5$, red, draw=red, circle [$_6$, blue, draw=blue] ] [$_6$, blue,
    draw=blue] ] [$_5$ [$_6$] ] [$_6$] ] [$_4$ [$_5$ [$_6$] ] [$_6$] ] [$_5$
    [$_6$] ] [$_6$] ]
  \end{forest}
  \caption{$\Upsilon$: Subsets of $\{1,2,3,4,5,6\}$}
  \label{fig:1}
\end{figure}

In the case at hand, the nodes in $\Upsilon$ can be decorated to encode
information about the sets of reflections in $R$ that they ultimately
represent. In Figure \ref{fig:1}, square nodes indicate dependent sets and
circled nodes indicate independent sets that contain a broken circuit. The
remaining nodes, by definition, correspond to the broken circuit basis, and
form a subtree, $\Upsilon_{\CB}$, shown in Figure \ref{fig:2}.

\begin{figure}[htb]
  \centering
  \begin{forest}
    for tree={l=15mm,s sep=2mm,inner sep=0mm,minimum size=4.3mm} [$_0$ [$_1$
    [$_3$ [$_6$] ] [$_4$ [$_6$] ] [$_6$] ] [$_2$ [$_4$ [$_6$] ] [$_5$ [$_6$]
    ] [$_6$] ] [$_3$ [$_5$ [$_6$] ] [$_6$] ] [$_4$ [$_5$ [$_6$] ] [$_6$] ]
    [$_5$ [$_6$] ] [$_6$] ]
  \end{forest}
  \caption{$\Upsilon_{\CB}$: Words in $\CB$}
  \label{fig:2}
\end{figure}

The tree $\Upsilon_{\CB}$ can alternately be constructed recursively by
successively adding the nodes with labels $1, 2, \dots, n$ to the tree
consisting of the root node $0$ only. Let us call
the \emph{tree at stage $m$} the
tree consisting of nodes with labels $0, 1, \dots, m$.  The tree at stage
$m$ is constructed from the tree at stage $m-1$ by checking, for \emph{each
  node} in the stage $m-1$ tree, whether it can be extended by a node with
label $m$, and if so, then by adding a node with label $m$.

To decide whether the node with word $i_1 \cdots i_k$ in the tree at stage
$m-1$ can be extended by a node labelled $m$, we use the following
observation: Suppose $i_1\cdots i_k$ is a word in $\CB$ with $i_k <
m$. Then the word $i_1\cdots i_k\, m$ \emph{contains} (we don't claim it
\emph{is}) a broken circuit if and only if there is an index $u>m$ such
that the word $i_1 \cdots i_k\, m\, u$ is dependent. Indeed, since
$i_1\cdots i_k$ does not contain a broken circuit, if $i_1\cdots i_k\,
m$ contains a broken circuit, then this broken circuit must contain
$m$.

In the example in Figure \ref{fig:2}, many subtrees appear repeatedly in the
tree $\Upsilon_{\CB}$ and carry redundant information. This suggests storing the information
in the form of a smaller directed acyclic graph with the property that
each rooted path in this smaller graph corresponds to a node in the original
tree with the same rooted path.

Such a graph $\Gamma$ is constructed from $\Upsilon_{\CB}$ by starting with
the leftmost maximal path in $\Upsilon_{\CB}$, then adjoining the other
maximal paths (say from left to right), then adjoining any missing paths of
length $l-1$, then adjoining any missing paths of length $l-2$, and so
on. Continuing the example of $W$ of type $A_3$, the graph $\Gamma$ has $9$
nodes (as opposed to $24$ in the orignal tree) and is given in Figure
\ref{fig:3}.

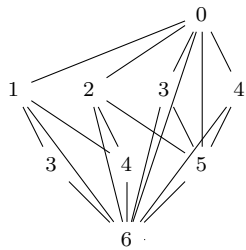
\begin{figure}[htb]
  \centering
  \begin{tikzpicture}[xscale=0.5,yscale=0.5]
    \node (0) at (2, 4) {$_0$}; \node (1) at (-3, 2) {$_1$}; \node (2) at
    (-1, 2) {$_2$}; \node (3) at (1, 2) {$_3$}; \node (4) at (3, 2) {$_4$};
    \node (31) at (-2, 0) {$_3$}; \node (41) at (0, 0) {$_4$}; \node (5) at
    (2, 0) {$_5$}; \node (6) at (0, -2) {$_6$}; \foreach \a/\b in
    {0/1,0/2,0/3,0/4,0/5,1/31,1/41,2/41,2/5,3/5,4/5} \draw (\a) -- (\b);
    \foreach \a in {0,1,2,3,4,5,6,31,41} \draw (\a) -- (6);
  \end{tikzpicture}
  \caption{The graph $\Gamma$ for $W$ of type $A_3$}
  \label{fig:3}
\end{figure}

Finally, the rooted paths in $\Gamma$ can be enumerated by a recursive depth
first traversal. Thus, the graph $\Gamma$ can be alternately be constructed
in the same fashion as $\Upsilon_{\CB}$, by successively adding the nodes
with label $1, 2, \dots, n$, and carefully tracking of the prefixes
represented by nodes with the same label.

Naturally, the graph $\Gamma$ depends on the chosen total order on $R$. In
the case of $E_8$, with the order of roots and reflections as produced by
CHEVIE, the graph $\Gamma$ has $1,207,608$ nodes and $15,552,964$ edges,
representing the $|W| = 696,729,600$ basis elements.

%%%%%%%%%%%%%%%%%%%%%%%%%%%%%%%%%%%%%%%%%%%%%%%%%%%%%%%%%%%%%%%%%%%%%%
%%%%%%%%%%%%% Acknowledgments
%%%%%%%%%%%%%%%%%%%%%%%%%%%%%%%%%%%%%%%%%%%%%%%%%%%%%%%%%%%%%%%%%%%%%%

\bigskip

\noindent {\bf Acknowledgments}: The research of this work was supported by
the Simons Foundation (Grant \#245399 to J.M.~Douglass)
and by
the DFG (Grant \#RO 1072/16-1 to G.~R\"ohrle).
J.M.~Douglass would like to acknowledge that
some of this material is based upon work supported by (while serving at) the
National Science Foundation.
Part of the research for this paper was carried out during a stay at
the Mathematical Research Institute Oberwolfach supported by the
``Research in Pairs'' program.

%%%%%%%%%%%%%%%%%%%%%%%%%%%%%%%%%%%%%%%%%%%%%%%%%%%%%%%%%%%%%%%%%%%%%%
%%%%%%%%%%%%% bibliography
%%%%%%%%%%%%%%%%%%%%%%%%%%%%%%%%%%%%%%%%%%%%%%%%%%%%%%%%%%%%%%%%%%%%%%
\bibliographystyle{plain}

%%%%%%%%%%%%%%%%%%%%%%%%%%%%%%%%%%%%%%%%%%%%%%%%%%%%%%%%%%%%%%%%%%%%%%
\end{document}